\documentclass{article}
\usepackage{amsmath,amssymb,amsthm}
\usepackage{amsfonts,latexsym} 
\usepackage{color}
\newcounter{satznum}
\newtheorem{theorem}{Theorem}[satznum]

\newtheorem{lemma}[theorem]{Lemma}

\newenvironment{acknowledgement}
 {\begin{trivlist}\item[]{\bf Acknowledgement.}}
 {\end{trivlist}}
\newenvironment{remark}
 {\begin{trivlist}\item[]{\bf Remark.}}
 {\end{trivlist}}

{\makeatletter
\gdef\nz{{\mathbb N}} 
\renewcommand{\P}{{\mathbb P}}
\gdef\nm{{\left\{1,\ldots,n\right\}}}

\begin{document}
\section*{Minimal clade size in the Bolthausen-Sznitman coalescent}
   {\sc F. Freund, A. Siri-J\'egousse}

   \begin{center}
      \tt \today \\
      
   \end{center}
\begin{abstract}
This article shows the asymptotics of distribution and moments of the size $X_n$ of the minimal clade of a randomly chosen individual in a Bolthausen-Sznitman $n$-coalescent for $n\to\infty$. The Bolthausen-Sznitman $n$-coalescent is a Markov process taking states in the set of partitions of $\left\{1,\ldots,n\right\}$, where $1,\ldots,n$ are referred to as individuals. The minimal clade of an individual is the equivalence class the individual is in at the time of the first coalescence event this individual participates in.\\
The main tool used is the connection of the Bolthausen-Sznitman $n$-coalescent with random recursive trees introduced by Goldschmidt and Martin (see \cite{goldschmidtmartin}). This connection shows that $X_n-1$ is distributed as the number $M_n$ of all individuals not in the equivalence class of individual 1 shortly before the time of the last coalescence event. Both functionals are distributed like the size $RT_{n-1}$ of an uniformly chosen table in a standard Chinese restaurant process with $n-1$ customers.We give exact formulae for these distributions.\\  
Using the asymptotics of $M_n$ shown by Goldschmidt and Martin in \cite{goldschmidtmartin}, we see $(\log n)^{-1}\log X_n$ converges in distribution to the uniform distribution on [0,1] for $n\to\infty$.\\  
We provide the complimentary information that $\frac{\log n}{n^k}E(X_n^k)\to \frac{1}{k}$ for $n\to\infty$, which is also true for $M_n$ and $RT_n$.

   \vspace{2mm}

   \noindent Keywords: minimal clade size, Bolthausen-Sznitman $n$-coalescent, Chinese restaurant process

   \vspace{2mm}

   \noindent AMS 2010 Mathematics Subject Classification:
            Primary 60C05;  Secondary 05C80, 60G09, 60F05, 60J27, 92D25; 
\end{abstract}
\subsection{Introduction} \label{sec:intro}
\setcounter{theorem}{0}

The Bolthausen-Sznitman $n$-coalescent is a time-homogeneous Markov process $(\Pi^{(n)}_t)_{t\geq 0}$ whose state space is the set of partitions of $\nm$. The only possible transitions in this process are those in which several blocks of a partition are merged (or coalesced) into one new block. Only one new block can be formed in a transition (no simultaneous mergers). Each $k$-tuple of $b$ present blocks is merged to a new block at rate $\frac{(k-2)!(b-k)!}{(b-1)!}$. The Bolthausen-Sznitman $n$-coalescent is a member of the $\Lambda$-$n$-coalescent family (which were introduced independently by Sagitov \cite{sagitov} and Pitman \cite{pitman}). A $\Lambda$-$n$-coalescent is again a  time-homogeneous, continuous-time Markov process whose state space is the set of partitions of $\nm$. The possible transitions are again mergers of multiple blocks into a new one. Each merger of $k$ blocks among $b$ happens with rate
$$\int_{[0,1]}x^{k-2}(1-x)^{n-k}\Lambda(dx)$$
for a finite measure $\Lambda$ on $[0,1]$. Note that the Bolthausen-Sznitman coalescent has $\Lambda=U_{[0,1]}$, the uniform distribution on $[0,1]$. 
Each $\Lambda$-$n$-coalescent can be represented as a random tree with $n$ leaves $\nm$ and random branch lengths by representing each merger as an internal node in the tree (the branch lengths are then the waiting times for the mergers, time is measured starting from the leaves). 
Also note that a $\Lambda$-$n$-coalescent at time $t$ forms a random exchangeable partition of $\nm$.

The Bolthausen-Sznitman $n$-coalescent was introduced by Bolthausen and Sznitman in 1998 (see \cite{boltsznit}). It has connections to population genetics and physics. In mathematical physics, it appears in the context of spin glasses (see \cite{boltsznit} and \cite{bovierkurkova}). It also seems to be a suitable model for the genealogy of a sample of $n$ alleles/genes/haplotypes in several models for selection in population genetics (see \cite{BDMM1}, \cite{BDMM2}, \cite{berestycki2schweinsb12}, \cite{NeherHallatschek}, \cite{Desai} see also the survey \cite{BrunetDerrida12}). Note that this is in contrast to the standard model for a genealogical tree of such a sample which is Kingman's $n$-coalescent ($\Lambda=\delta_0$, only 2 merger at a time, introduced in \cite{kingman}).  Also note that due to the interpretation of the Bolthausen-Sznitman $n$-coalescent as a genealogical tree, we refer to $\left\{1,\ldots, n\right\}$ as individuals.\\
Here, we focus on the Bolthausen-Sznitman $n$-coalescent as a model for a genealogical tree which depicts the ancestry of $n$ alleles sampled at a genetic locus. 
Since the genealogical tree often is endowed with a mutation structure which is interpreted under the infinitely-many sites model, we assume a locus consisting of many nucleotide sites, for example a gene. Different alleles can thus also be seen as different haplotypes at the according sites. One important information coded in the genealogy is the relatedness of an allele randomly chosen from the sample to the rest of the sample. There are two functionals/statistics of the genealogical tree which transport complementary information about this relatedness.
The first functional is the length $E_n$ of an external branch chosen at random from the $n$ external branches associated with the leaves $\nm$ of the tree, introduced by Fu and Li in \cite{fuli}. $E_n$ gives the time that the chosen allele has to evolve independently of the rest of the sample (e.g., by mutation). This gives a measure of the genetic uniqueness of this allele relative to the rest of the sample. The second functional is the size $X_n$ of the minimal clade containing the randomly chosen allele, introduced by Blum and Fran\c{c}ois in \cite{blumfrancois}. The minimal clade can be defined in different, yet equivalent ways: The minimal clade is 
\begin{itemize}
\item the equivalence class that contains the (randomly chosen) allele $i\in\nm$ at the first time $i$ was merged,
\item all leaves of the subtree rooted at the most recent ancestor of allele $i$,
\item all descendants of the most recent ancestor of allele $i$.
\end{itemize}
The minimal clade can also be seen as the smallest family containing $i$. The size of the minimal clade gives the complementary information how many individuals share the genealogy with allele $i$ "after" time $E_n$ (note that since we measure time from leaves to root, "after" $E_n$ actually means further back in time). 

The external branch length is already analyzed well for several $\Lambda$-$n$-coalescents in the literature. Its distribution follows a recursion and its asymptotics for sample size $n\to\infty$ are known for various $\Lambda$-$n$-coalescents (see \cite{freundmoehle1}, \cite{caliebe}, \cite{blumfrancois}, \cite{gneiksmoe}, \cite{DFSY12}).
 For the minimal clade size, though, only results for Kingman's $n$-coalescent
 ($\Lambda=\delta_0$, only 2 merger at a time)
  are known (including asymptotics for $n\to\infty$, see \cite{blumfrancois}).

The purpose of this paper is to analyze the distribution of the minimal clade size $X_n$ in the case of the Bolthausen-Sznitman $n$-coalescent and its asymptotics for sample size $n\to\infty$. We will exploit the construction of the Bolthausen-Sznitman $n$-coalescent using a random recursive tree introduced by Goldschmidt and Martin (see \cite{goldschmidtmartin}) to prove our results. First, we observe that this construction yields that the process describing the set of relatives of a randomly chosen  individual in the Bolthausen-Sznitman $n$-coalescent  process (which is its equivalence class without the individual itself) is equal in law to the time-reversed process describing the set of non-relatives of the chosen individual (all individuals in different equivalence classes than the chosen individual). This shows that the minimal clade size actually is distributed as the sum $M_n$ of the sizes of all blocks not containing 1 which participate in the last collision in the $n$-coalescent. Convergence in distribution of properly scaled $M_n$ for $n\to\infty$ was shown already by Goldschmidt and Martin in \cite{goldschmidtmartin} and thus the same asymptotic behavior holds for $X_n$, namely $(\log n)^{-1}\log X_n$ converges in distribution to the uniform distribution on [0,1].\\      
Note that due to the connection between the random recursive tree and the standard Chinese Restaurant process, we observe that $X_{n}-1$ and $M_n$ are distributed as the size of a uniformly chosen table (not chosen by a size-biased pick!) in the Chinese restaurant process (again for $M_n$ in accordance to \cite{goldschmidtmartin}). This allows us to give several formulae for the exact distribution of $X_n$. Using these, we show that $\frac{\log n}{n^k}E(X_n^k)\to \frac{1}{k}$ for $n\to\infty$, which gives complementary information to the weak convergence result.

\subsection{Minimal clade size in the Bolthausen-Sznitman $n$-coalescent} \label{sec:boltsznit}
\setcounter{theorem}{0}
Set $[n]:=\nm$ and $[n]_0:=\{0,\ldots,n\}$. For a partition $\eta$ of $[n]$, let $C_i(\eta)$ denote the equivalence class of $i\in[n]$ and $|C_i(\eta)|$ its size. Let  $(\Pi^{(n)}_t)_{t\geq 0}$ be a $\Lambda$-$n$-coalescent. Since we want to look at the minimal clade size of a randomly chosen allele in the sample whose genealogy is given by $(\Pi^{(n)}_t)_{t\geq 0}$, define $I$ as a uniform pick from $[n]$ independent of the $n$-coalescent. Now, first define the length of a randomly chosen external branch (associated with the randomly chosen $I\in[n]$) by
$$E_n:=\inf\{t\geq 0 , C_I(\Pi^{(n)}_t)\neq \{ I\} \}.$$ 
 Now we define the size of the minimal clade of the randomly chosen allele $I$ as 
\begin{equation}\label{def:mincladesize}
X_n:=|C_I(\Pi^{(n)}_{E_n})|.
\end{equation}
Note that, due to exchangeability, we don't change the distributions of $E_n$ and $X_n$ if we assume $I=1$. Also note that due to the interpretation of a $n$-coalescent as a genealogical tree, we refer to $\left\{1,\ldots,n\right\}$ as individuals.\\
From now on, we will abbreviate the size of the minimal clade of $I=1$ with minimal clade size.\\
The minimal clade of individual 1 is the size of the equivalence class of 1 at the first coalescence event that the individual participates in. In \cite{goldschmidtmartin}, Goldschmidt and Martin have analysed the behavior of the total mass $M_n$ of the equivalence classes not containing 1 at the last coalescence event in the Bolthausen-Sznitman $n$-coalescent (see \cite[Thm. 3.1]{goldschmidtmartin}). Note that $M_n$ can also be written as $n-|C_1(\Pi^{(n)}_{\tau_n-})|$, where $\tau_n$ is the waiting time for the last coalescence event. Both $X_n$ and $M_n$ are functionals of the equivalence class of 1 in the Bolthausen-Sznitman $n$-coalescent at different times. Thus, it's interesting how the equivalence class of 1 changes over time. It will only grow by merging with other equivalence classes at coalescence times, but not necessarily at all coalescence times. We define $S^{(n)}_i$ as the equivalence class of 1 after the $i$th merging event which 1 participates in. What are the properties of $(S^{(n)}_i)_{i\in[\kappa_n]_0}$, where $\kappa_n$ is the number of merging events 1 participates in? The results from \cite{goldschmidtmartin} answer this question. There, the authors show a construction of the Bolthausen-Sznitman $n$-coalescent by applying a cutting procedure to a random recursive tree and use it, among other questions, to analyse $M_n$.\\  
We will show in detail that this construction enables us to analyse the behaviour of $S^{(n)}$ and that it can be expressed in terms of a Chinese retaurant process. Note that this is just the line of reasoning from \cite{goldschmidtmartin}. Let's quickly recall the construction of the Bolthausen-Sznitman $n$-coalescent from a random recursive tree from \cite[Prop. 2.2]{goldschmidtmartin} as well as the connection to the Chinese restaurant process. Here, we give a simplified version just constructing the jump chain of the $n$-coalescent.\\
We start with a random recursive tree with $n$ vertices, i.e. a uniformly distributed random variable on the set of all recursive trees with $n$ vertices $1,\ldots,n$ rooted in 1 (here, the branches carry no length information). Now construct the jump chain as follows.
\begin{enumerate}
\item Choose an edge at random
\item Cut the tree at this edge. All labels that are in the subtree not containing the root are added to the node of the subtree containing the root which was adjacent to the cut edge.
\item Define a partition by taking the labels at each node of the subtree containing the root. This partition has the same law as the Bolthausen-Sznitman $n$-coalescent after the first jump.
\item Repeat subsequently steps 1-3 with the subtree containing the root. This leads to partitions which have the same law as the Bolthausen-Sznitman $n$-coalescent after the 2nd, 3rd, $\ldots$ jump.  
\end{enumerate} 
We now come to the connection of the random recursive tree with the Chinese restaurant process.\\
First, recall that the standard Chinese restaurant process is a sequential construction of a uniform permutation of $[n]$. Imagine a restaurant with tables $1,2,3,\ldots$ with infinitely many chairs. $n$ customers $1,\ldots,n$ sit down at the tables after the following rule:
\begin{itemize}
\item customer 1 sits at table 1,
\item if $i-1$ customers have taken their seat, the $i$th customer sits with equal probability at one of the following $i$ places: 
\begin{itemize}
\item on a chair directly to the left of an already seated customer (possibly between customers),  
\item at a previously unoccupied table. 
\end{itemize}      
\end{itemize}
Writing down the customers at each table in seating order, we get the cycles of a uniform random permutation of $[n]$. If we only record the customers at each table, but not the seating order, we get an exchangeable partition of $[n]$ whose distribution is given by Ewens sampling formula. More information on this process can be found in \cite[Ch. 3.1]{pitmancombstoch}. We will abbreviate a standard Chinese restaurant process with $n$ customers by $CRP(n)$.\\
A $CRP(n-1)$ can be found in a random recursive tree with $n$ vertices in the following way (see \cite[p. 724-725]{goldschmidtmartin}). We define a subtree of '1' in the random recursive tree as a rooted subtree whose root is adjacent (connected by one edge) to the root '1' of the whole tree. Then the subtrees of '1' form a exchangeable partition of $\left\{2,\ldots,n\right\}$ which can be described as a $CRP(n-1)$ with customers labelled $2,\ldots,n$. The following lemma just is a write-up of the line of reasoning from \cite[p. 725]{goldschmidtmartin} and gives a discrete analogon of a part of \cite[Cor. 16]{pitman} (in \cite{goldschmidtmartin}, the line of reasoning presented here is a part of an alternative proof for \cite[Cor. 16]{pitman})
\begin{lemma}(practically from Goldschmidt, Martin)\label{lem:blockof1}  
Let $\kappa_n$ be the number of collisions in a Bolthausen-Sznitman $n$-coalescent individual 1 participates in. For $i\in[\kappa_n]_0$, let $S^{(n)}_i$  be the equivalence class of 1 in the Bolthausen-Sznitman $n$-coalescent after the $i$th collision. For a $CRP(n-1)$ with $K_{n-1}$ tables, let $RT_1,\ldots,RT_{K_{n-1}}$ be the tables in random order. Then $S^{(n)}=(S^{(n)}_i)_{i\in[\kappa_n]_0}$ is distributed as $(\left\{1\right\}\cup\bigcup_{j\in [i]}RT_j)_{i\in[K_{n-1}]}$.\\ 
Moreover, the process $S^{(n)}\setminus \left\{1 \right\}=(S^{(n)}_i\setminus \left\{1 \right\})_{i\in[\kappa_n]_0}$ giving the relatives of individual 1 through time is distributed as the time-reversed process $[n]\setminus S^{(n)}=([n]\setminus S^{(n)}_{\kappa_n-i})_{i\in[\kappa_n]_0}$ giving the nonrelatives of individual 1.    
\end{lemma}
If the Bolthausen-Sznitman $n$-coalescent is constructed via cutting a random recursive tree, this lemma can be described more graphically: The equivalence class of 1 grows by adding tables chosen uniformly at random from the Chinese restaurant process  with $n-1$ customers given by the subtrees of '1' in the random recursive tree.\\
Note that we actually chose tables at random not individuals sitting at tables, so we don't make size-biased picks.  
\begin{proof}
We construct the Bolthausen-Sznitman $n$-coalescent via cutting a random recursive tree. The equivalence class of 1 is merged with other equivalence classes as soon as an edge adjacent to  the root is cut in the random recursive tree. The equivalence class of 1 is then merged with the subtree of '1' which is connected by that edge. Since the edges are chosen at random, this means that a uniformly chosen table of the $CRP(n-1)$ given by the subtrees of '1' is merged with the class of 1.
\end{proof}
Since $X_n-1=|S^{(n)}_1\setminus \left\{1 \right\}|$ and $M_n=|[n]\setminus S^{(n)}_{\kappa_n-1}|$, Lemma \ref{lem:blockof1} shows that $X_n-1$ and $M_n$ have the same distribution, namely that both are distributed as the size of a uniformly chosen table in a $CRP(n-1)$. This means that the known results for the asymptotics of $M_n$  which are given in \cite[Thm. 3.1]{goldschmidtmartin} are valid for $X_n-1$ and due to a Slutski argument are also valid for $X_n$. 
\begin{theorem}\label{thm:minclade}
Let $n\in \{ 2,3\,\ldots\}$. Let $X_n$ be the minimal clade size  in the Bolthausen-Sznitman $n$-coalescent. $X_n$ is  distributed on $2,\ldots,n$. $X_n$ is distributed as the size of a randomly chosen table in a $CRP(n-1)$ reduced by 1 and 
$$\frac{\log X_n}{\log n} \rightarrow U_{[0,1]}$$
holds in distribution for $n\to\infty$, where $U_{[0,1]}$ is the uniform distribution on $[0,1]$.
\end{theorem}  
Additionally to this result, we give the complementary information of the
exact law of $X_n$ and of the first order behaviour of all moments of $X_n$ for $n\to\infty$. For this, we need more knowledge about the distribution of $X_n$.

Theorem \ref{thm:minclade} states that the distribution of $X_n$ can be expressed in terms of the Chinese restaurant process. We will use this to derive three formulae for the distribution of $X_n$. Let's recall two possibilities to look at the distribution of customers at tables in a $CRP(n)$. It is well known that this distribution in a $CRP(n)$ is given by the celebrated Ewens sampling formula with mutation parameter $\theta=1$ (e.g., see \cite[eq. 1.3]{arratiabarbourtavare}). We use two different possibilities to look at the Ewens sampling formula in equations \eqref{eq:ESF1} and \eqref{eq:ESF2}. First, we can record how many tables in a $CRP(n)$ have exactly $i$ customers, which we denote by $A^{(n)}_i$, for each $i\in[n]$. Then for $a_1,\ldots,a_n\in[n]_0$ with $\sum_{i\in[n]}ia_i=n$, we have
\begin{equation}\label{eq:ESF1}
\P(A^{(n)}_1=a_1, \dots,A^{(n)}_n=a_{n})=\prod_{i=1}^{n}\frac{1}{a_i!i^{a_i}}.
\end{equation} 
On the other hand, we can record the probability that certain sets of customers sit at tables $1,2,\ldots$ (this forms a partition $\eta$ of $[n]$). The probability that we find a certain partition $\eta$ 
(with blocks ordered by their least element)
of $[n]$ with $k$ occupied tables and $n_i$ customers at the $i$th occupied table is
\begin{equation}\label{eq:ESF2}
P(CRP(n)=\eta)=\frac{1}{n!}\prod_{i\in[k]}(n_i-1)!.
\end{equation}
This leads to several possibilities to express the distribution of $X_n$.     
\begin{lemma}\label{lem:distclade}
Let $n\in\left\{2,3,...\right\}$. Let $X_n$ be the minimal clade size in a Bolthausen-Sznitman $n$-coalescent. For $m\in\nz$, let $A^{(m)}_i$ be the number of tables with exactly $i$ customers in a $CRP(m)$ and $K_{m}=\sum_{i\in[n-1]}A^{(m)}_i$ the number of occupied tables. Define $K_0=0$ ('empty restaurant') Then for $j\in[n-1]$ 
\begin{itemize}
\item[a)] Denoting 
$\Gamma_{n}=\{a_1,\ldots,a_{n-1}\in[n-1]_0 , \sum_{i=1}^{n-1}ia_i=n-1\}$
		\begin{align*}
			P(X_n=j+1)&=E\left(\frac{A^{(n-1)}_j}{K_{n-1}}\right)\\
		      		  &=\sum_{\Gamma_{n}}\frac{a_j}{\sum_{i=1}^{n-1}a_i}\prod_{i=1}^{n-1}\frac{1}{a_i!i^{a_i}}.
		\end{align*}      
		      
\item[b)] Denoting
$\Delta({n},k)=\{n_1,\ldots ,n_k\in[n], \sum_{i=1}^k n_i=n\}$ for $k\leq n$,
$$P(X_n=j+1)= 
 \frac{1}{j}\sum_{k=1}^{n-1-j}\frac{1}{(k+1)!}\sum_{\Delta(n-1-j,k)}\frac{1}{n_1\cdots n_k} $$ 
 for $j<n-1$
 and $P(X_n=n)=\frac{1}{n-1}$.

\item[c)] Let $B_1,B_2,\dots$ be independent Bernoulli-distributed random variables with success probability $\frac{1}{i}$ for $B_i$.
 \begin{align*}
				P(X_n=j+1)&=\frac{1}{j}E\left(\frac{1}{1+K_{n-1-j}}\right)\\
				&=\frac{1}{j}E\left(\frac{1}{1+\sum^{n-1-j}_{i=1}B_i}\right),
		  \end{align*} 
		  
\end{itemize}
\end{lemma} 
Note that above lemma also holds true for $M_n$ and the size $RT_{n-1}$ of a randomly chosen table in a $CRP(n-1)$ (just replace $j+1$ with $j$).
Also note that this result provides a very rare example where an exact law is obtained for a functional of an exchangeable non-Kingman, non-starshaped $n$-coalescent.
\begin{proof}
Due to Theorem \ref{thm:minclade}, we know that $X_n-1$ is distributed as the size of a randomly chosen table in a $CRP(n-1)$. Given the table counts $A^{(n-1)}_1,\ldots,A^{(n-1)}_{n-1}$, the probability that we randomly choose a table with $j$ customers is $\frac{A^{(n-1)}_j}{\sum_{i=1}^{n-1}A^{n-1}_i}=\frac{A^{(n-1)}_j}{K_{n-1}}$. Summing over the distribution of the table counts given by (\ref{eq:ESF1}) gives a).\\
Now look at the partition $\eta$ of $[n]$ constructed via a $CRP(n-1)$ whose distribution is given by \eqref{eq:ESF2}. We are interested in the partition not in order of least elements, but in exchangeable order (meaning that if the partition has $k$ blocks, we order them randomly). Let $N^{(n-1)}_1,\ldots, N^{(n-1)}_k$ be the table sizes in exchangeable order. By combinatorial arguments (see \cite[(2.7)]{pitmancombstoch}), we get
\begin{align}\label{eq:ESFexch}
P(N_1^{(n-1)}=n_1,\ldots,N^{(n-1)}_k=n_k)&=\binom{n-1}{n_1,\ldots,n_k}\frac{1}{k!}\underbrace{\frac{1}{(n-1)!}\prod_{i=1}^k(n_i-1)!}_{\mbox{prob. of $\eta$, least elements}}\nonumber\\
&=\frac{1}{k!\prod_{i=1}^k n_i}.
\end{align}  
 The size of a randomly picked table in the CRP is distributed as $N^{(n-1)}_1$. This is just the marginal distribution from above formula, namely
\begin{equation}\label{eq:prob2clade}
P(X_n=j+1)=P(N^{(n-1)}_1=j)
=\sum_{k=1}^{n-1}\frac{1}{k!}\sum_{n_2,\ldots ,n_k\in[n-1]\atop j+\sum_{i=2}^n n_i=n-1}\frac{1}{j\cdot n_2\cdots n_k}.
\end{equation}
If $j=n-1$, \eqref{eq:prob2clade} equals $\frac{1}{n-1}$. For $1\leq j<n-1$, we have
\begin{align*}
P(X_n=j+1)&=\sum_{k=2}^{n-j}\frac{1}{k!j}\sum_{n_2,\ldots ,n_k\in[n-1-j]\atop \sum_{i=2}^k n_i=n-1-j}\frac{1}{n_2\cdots n_k}\\
&=\frac{1}{j}\sum_{k=1}^{n-1-j}\frac{1}{(k+1)!}\sum_{\Delta(n-1-j,k)}\frac{1}{n_1\cdots n_k},
\end{align*}
where the last equation is due to an index shift. This shows b).\\
 To show c), we compare \eqref{eq:prob2clade} with $E((1+K_{n-1-j})^{-1})$. First note that for $j=n-1$, we have $K_0=0$ and thus $$\frac{1}{n-1}E\left(\frac{1}{K_0+1}\right)=\frac{1}{n-1},$$ which matches the expression in b). Now assume $1\leq j<n-1$. If we look at the table sizes in exchangeable order, we can compute $P(K_{n-1-j}=k)$ by summing up the probabilities of all possible configurations of table sizes of exactly $k$ occupied tables in a $CRP(n-1-j)$. Using \eqref{eq:ESFexch}, this leads to  
\begin{align*}
 E\left(\frac{1}{1+K_{n-1-j}}\right)&=\sum_{k=1}^{n-1-j} \frac{1}{k+1}\sum_{\Delta(n-1-j,k)}\frac{1}{k!n_1\cdots n_k}
 \\&=\sum_{k=1}^{n-1-j} \frac{1}{(k+1)!}\sum_{\Delta(n-1-j,k)}\frac{1}{n_1\cdots n_k}.
\end{align*}   

Comparison with \eqref{eq:prob2clade} yields 
$$P(X_n=j+1)=\frac{1}{j}E\left(\frac{1}{1+K_{n-1-j}}\right).$$
Recall that $K_{n-1-j}$ is distributed as the number of cycles in a uniform permutation of $[n-1-j]$. It is well-known that the number of cycles is distributed as the sum of
 independent Bernoulli variables $B_1,\ldots, B_{n-1-j}$ with success probability $\frac{1}{i}$ for $B_i$ (e.g., see \cite[p.10]{arratiabarbourtavare}). This proves c).  
\end{proof}
\begin{remark}
Let $K_n$ be the number of occupied tables in a $CRP(n)$. Using $K_n\stackrel{d}{=}\sum_{i\in[n]}B_i$ for independent Bernoulli variables with success probability $\frac{1}{i}$, we deduce the recursion
$$E\left(\frac{1}{m+K_{n}}\right)=(1-\frac{1}{i})E\left(\frac{1}{m+K_{n-1}}\right)+\frac{1}{i}E\left(\frac{1}{m+1+K_{n-1}}\right)$$
for all $m\in\nz_0$. This recursion gives an efficient method to compute the distribution of the minimal clade size $X_n$ by using the representation in Lemma \ref{lem:distclade} c).  
\end{remark}
\begin{remark}
In \cite{goldschmidtmartin}, Goldschmidt and Martin have proven the weak convergence result for $M_n$ for $n\to\infty$ by using the construction of the Bolthausen-Sznitman $n$-coalescent via cutting a random recursive tree and embedding the random recursive tree in a Yule process. However, as also hinted at by Goldschmidt and Martin (see \cite[Cor. 3.3, Remark a)]{goldschmidtmartin}), the representation of $M_n$ as a uniformly chosen table in a $CRP(n-1)$ allows to use results about uniform random permutations to prove the convergence part of \ref{thm:minclade} without using the Yule process embedding.
\begin{proof} \textit{(Alternative proof of \ref{thm:minclade})}\\  
First, let's look at the distribution function of $\frac{\log(X_n-1)}{\log (n-1)}$. Let $x\in[0,1]$. Using Lemma \ref{lem:distclade} a), we get 
\begin{align}\label{eq:cladesizebyESF}
 P\left(\frac{\log(X_n-1)}{\log (n-1)}\leq x\right)&=P(X_n-1\leq (n-1)^x)=\sum^{\lfloor (n-1)^x\rfloor}_{j=1}E\left(\frac{A^{(n-1)}_j}{K_{n-1}}\right)\nonumber \\
&=E\left(\frac{\sum^{\lfloor (n-1)^x\rfloor}_{j=1}A^{(n-1)}_j}{\sum_{i=1}^{n-1}A^{(n-1)}_i}\right),
\end{align} 
where $A_i^{(n-1)}$ is the number of tables with exactly $i$ customers in a $CRP(n-1)$.
 The functional central limit theorem of DeLaurentis and Pittel \cite{delaurentispittel} (see also Hansen \cite{hansen}) states
 $$\left(\frac{\sum^{\lfloor n^x\rfloor}_{j=1}A^{(n)}_j-x\log n}{\sqrt{\log n}}\right)_{x\in[0,1]}\stackrel{d}{\to}(B_x)_{x\in[0,1]},$$
 in $D[0,1]$ when $n\to\infty$, where $B$ is a standard Brownian motion. This implies 
$$\frac{\sum^{\lfloor n^x\rfloor}_{j=1}A^{(n)}_j}{\log n}\stackrel{p}{\to}x$$
for $x\in[0,1]$. We apply this result to both the nominator and denominator of the right hand side of (\ref{eq:cladesizebyESF}) (inside of $E(\cdot)$) and get
$$\frac{\sum^{\lfloor (n-1)^x\rfloor}_{j=1}A^{(n-1)}_j}{\sum_{i=1}^{n-1}A^{(n-1)}_i}=\frac{\sum^{\lfloor (n-1)^x\rfloor}_{j=1}A^{(n-1)}_j}{\log (n-1)}\frac{\log (n-1)}{\sum_{i=1}^{n-1}A^{(n-1)}_i}\stackrel{p}{\to} x$$ for $n\to\infty$. Since $0\leq \frac{\sum^{\lfloor (n-1)^x\rfloor}_{j=1}A^{(n-1)}_j}{\sum_{i=1}^{n-1}A^{n-1}_i}\leq 1$ for all $x,n$, we have uniform integrability and hence
$$E\left(\frac{\sum^{\lfloor (n-1)^x\rfloor}_{j=1}A^{(n-1)}_j}{\sum_{i=1}^{n-1}A^{(n-1)}_i}\right)\to x$$ for $n\to\infty$ which shows
$$\frac{\log(X_n-1)}{\log (n-1)}\stackrel{d}{\to} U_{[0,1]},$$ where $U_{[0,1]}$ is the uniform distribution on $[0,1]$. $\frac{\log(X_n)}{\log n}$ behaves in the same way which can be shown with a Slutski argument.
\end{proof}
\end{remark}
For the asymptotics of moments of $X_n$ (as well as $M_n$ and $RT_{n}$), we use the expression for $P_{X_n}$ from Lemma \ref{lem:distclade} c), namely
$$P(X_n=j+1)=\frac{1}{j}E\left(\frac{1}{1+K_{n-1-j}}\right),$$
where $K_n$ is the number of occupied tables in a $CRP(n)$, we will be ab. Note that $K_n$ also gives the number of cycles in a uniform permutation of $\nm$. Thus, the distribution of $K_n$ is given by
\begin{equation}\label{eq:distknstirling}
P(K_n=k)=\frac{S_{n,k}}{n!} \ \mbox{   for }k\in [n],
\end{equation}
where $(S_{n,k})_{k\in[n],n\in\nz}$ denote the absolute Stirling numbers of the first kind.\\
It is well-known that (see, e.g., \cite[Eq. 3.2]{pitmancombstoch})
\begin{equation}\label{eq:knstronglaw}
\frac{K_n}{\log n}\to 1 \ \ \mbox{almost surely}
\end{equation}
for $n\to\infty$. Since we want to use Lemma \ref{lem:distclade} c), we're more interested in the behaviour of $E((1+K_n)^{-1})$. From (\ref{eq:knstronglaw}), we immediately get
\begin{equation}\label{eq:kninvstronglaw}
\frac{\log n}{1+K_n}\to 1 \ \ \mbox{almost surely}
\end{equation}   
for $n\to\infty$.
We will need a $L^1$-version of (\ref{eq:kninvstronglaw}).
\begin{lemma}\label{lem:knl1}
$$\frac{\log n}{1+K_n}\to 1 \ \ \mbox{in $L^1$ for $n\to\infty$}.$$
\end{lemma}
\begin{proof}
The result follows from  (\ref{eq:kninvstronglaw}) and the uniform integrability of $\frac{\log n}{1+K_n}$, which we show now. Note that since $\frac{\log n}{1+K_n}\leq \frac{\log n}{K_n}$ for all $n\in\nz$, it suffices to show uniform integrability for $\frac{\log n}{K_n}$. Let $A>0$ and $H_n\stackrel{d}{=}Pn(\log n)$ be a Poisson-distributed random variable with parameter $\log n$. Note that $H_n\stackrel{d}{=}\sum_{i\in[\log n]}H^{(1)}_i$, where $(H^{(1)}_i)_{i\in\nz}$ are i.i.d. with $H^{(1)}_1\stackrel{d}{=}Pn(1)$. For $A>1$, we have
\begin{align*}
\int_{\left\{\frac{\log n}{K_n}\geq A\right\}}\left|\frac{\log n}{K_n}\right| dP &= \log n\sum_{k=1}^{A^{-1}\log n}\frac{1}{k}P(K_n=k)\\
&\stackrel{(\ref{eq:distknstirling})}{=}\log n\sum_{k=1}^{A^{-1}\log n}\frac{S_{n,k}}{n!k}\\
&=\sum_{k=1}^{A^{-1}\log n}\frac{(\log n)^k}{k!}e^{-\log n}\left(\frac{1}{\Gamma(1+r)}+O\left(\frac{k}{(\log n)^2}\right)\right)\\
&\leq CP\left(H_n\leq \frac{\log n}{A}\right)
\\&=P\left(\frac{\sum_{i\in\log n}H^{(1)}_i}{\log n}\leq A^{-1}\right)\to 0
\end{align*}
for $n\to\infty$, where $r=(k-1)(\log n)^{-1}$ and $C$ is a suitable constant. Here, we use the uniform asymptotic expansion from Hwang (see Theorem 2 in \cite{Hwang}) for the absolute Stirling numbers $S_{n,k}$ of the first kind for $1\leq k\leq A^{-1}\log n$ (we actually use the cruder version from \cite[Eq. 1.30]{arratiabarbourtavare}). The convergence to 0 follows from the law of large numbers for $(H^{(1)}_i)_{i\in\nz}$.\\
This computation shows the uniform integrability of $\frac{\log n}{1+K_n}$ and thus the lemma. 
\end{proof}

\begin{theorem}
For $n\in\left\{2,3,...\right\}$, let $X_n$ be the minimal clade size in the Bolthausen-Sznitman $n$-coalescent. For all $k\in\nz$, we have
$$\frac{\log n}{n^k}E(X_n^k)\to \frac{1}{k}$$ 
for $n\to\infty$.\\
\end{theorem}
Again, this theorem is also true for $M_n$ and $RT_n$ instead of $X_n$.
\begin{proof}
Using Lemma \ref{lem:distclade} c), we get
\begin{align*}
E((X_n-1)^k)&=\sum_{j=1}^{n-1}j^{k-1}E\left(\frac{1}{1+K_{n-1-j}}\right)\\
&=\sum_{l=0}^{n-2}(n-1-l)^{k-1}E\left(\frac{1}{1+K_l}\right)\\
&=\sum_{i=0}^{k-1}\binom{k-1}{i}(n-1)^{k-1-i}(-1)^i\sum_{l=0}^{n-2}l^iE\left(\frac{1}{1+K_l}\right)
\end{align*}
We will now use Karamata's Tauberian theorem for power series (see \cite[Corr. 1.7.3]{binghamteugelsgoldie}). It states (among other things) that if $a_l\sim \frac{c}{\Gamma(\rho)}l^{\rho-1}\mathcal{L}(l)$ for $n\to\infty$, where $c,\rho>0$ and $\mathcal{L}$ is a slowly varying function, then $\sum_{k\in[n]}a_k\sim \frac{c}{\Gamma(1+\rho)}n^{\rho}\mathcal{L}(n)$. We define $a_l:=l^iE\left(\frac{1}{1+K_l}\right)$. Note that $a_l\sim \frac{l^i}{\log l}$ for $l\to\infty$ due to Lemma \ref{lem:knl1}, which enables us to use the Tauberian theorem for $a_l$ with $c:=\Gamma(i+1)=i!$, $\rho=i+1$ and $\mathcal{L}(n)=(\log n)^{-1}$. For each $i\in[k-1]_{0}$, we thus have  
$$\sum_{l=0}^{n-2}l^iE\left(\frac{1}{1+K_l}\right)\sim \frac{1}{i+1} \frac{n^{i+1}}{\log n}$$ for $n\to\infty$. This shows 
\begin{align*}
\frac{\log n}{n^{k}}E((X_n-1)^k)&= \sum_{i=0}^{k-1}\binom{k-1}{i}(n-1)^{k-1-i}\frac{\log n}{n^{k}}(-1)^i\sum_{l=0}^{n-2}l^iE\left(\frac{1}{1+K_l}\right)\\
&\sim \sum_{i=0}^{k-1}\binom{k-1}{i}\frac{\log n}{n^{i+1}}(-1)^i\frac{1}{i+1} \frac{n^{i+1}}{\log n}\\
&=\sum_{i=0}^{k-1}\binom{k-1}{i}\frac{(-1)^i}{i+1}=\frac{1}{k}
\end{align*}
for $n\to\infty$, where the last equation follows by elementary calculations. Thus, for each $k\in\nz$, we have established $$\frac{\log n}{n^k}E((X_n-1)^k)\to \frac{1}{k}$$
for $n\to\infty$.
The theorem now is proven as
$$\frac{\log n}{n^k}E(X_n^k)=\sum_{i\in[k]_0}\binom{k}{i}\frac{\log n}{n^k}E((X_n-1)^i)\to \frac{1}{k}$$
for $n\to\infty$
\end{proof}

\begin{remark}
 This last result fits well with the notion that $X_n$ can heuristically be seen as $n^
U$ when $n$ is big ($U$ uniformly distributed on $[0,1]$) following from Theorem \ref{thm:minclade}, since the $k$th moment of $n^U$ is $\frac{n^k}{k\log n}$.\\
We compare this heuristic to the results for $X_n$ in Kingman's $n$-coalescent from \cite[p.4]{blumfrancois}, where the authors state that $X_n$, without scaling, converges to a Yule distribution of parameter $\rho=2$. So in the Bolthausen-Sznitman $n$-coalescent, the minimal clade size is much bigger than in Kingman's coalescent. This agrees with the more starlike shape of a non-Kingman $n$-coalescent compared to Kingman's $n$-coalescent.
\end{remark}

\begin{acknowledgement}
We thank an anymonous referee to point out the connection of the minimal clade with the non-relatives of 1 in the last coalescence event of the Bolthausen-Sznitman $n$-coalescent and hence the proof of the convergence part of Theorem \ref{thm:minclade}. We thank J.-S. Dhersin, M. M\"ohle and L. Yuan for a fruitful discussion concerning the alternative proof of Theorem \ref{thm:minclade} and also E. Teufl for advice concerning which journal to submit to.   
\end{acknowledgement}

Fabian Freund

Crop Plant Biodiversity and Breeding Informatics Group (350b), Institute of Plant Breeding, Seed Science and Population Genetics, University of Hohenheim, Fruwirthstrasse 21, 70599 Stuttgart, Germany

Email address: fabian.freund@uni-hohenheim.de
\\

Arno Siri-J\'egousse

CIMAT, A.C., Calle Jalisco s/n, Col. Mineral de Valenciana, 36240 Guanajuato, Guanajuato, Mexico

Email address: arno@cimat.mx
 \end{document}